\documentclass[10pt]{amsart}

\usepackage{amsmath,amssymb,amscd,amsthm,mathtools,verbatim}
\newtheorem{thm}{Theorem}[section]
\newtheorem{lem}[thm]{Lemma}
\newtheorem{prop}[thm]{Proposition}
\newtheorem{cor}[thm]{Corollary}
\newtheorem{rmk}[thm]{Remark}
\newtheorem{definition}[thm]{Definition}

\newcommand{\CC}{\mathbb{C}}

\newcommand{\HH}{\mathbb{H}}
\newcommand{\QQ}{\mathbb{Q}}
\newcommand{\PP}{\mathbb{P}}
\newcommand{\ZZ}{\mathbb{Z}}
\newcommand{\A}[1]{\mathcal{A}_{#1}}
\newcommand{\M}[1]{\mathcal{M}_{#1}}
\newcommand{\CA}[1]{\overline{\mathcal{A}_{#1}}}
\newcommand{\CM}[1]{\overline{\mathcal{M}_{#1}}}

\newcommand{\CHF}{\overline{HF}}
\newcommand{\HF}{\mathcal{HF}}
\newcommand{\ep}{\epsilon}

\DeclareMathOperator{\Cat}{Cat}
\DeclareMathOperator{\Mat}{Mat}
\DeclareMathOperator{\Ima}{Im}
\DeclareMathOperator{\Sp}{Sp}
\DeclareMathOperator{\diag}{diag}
\DeclareMathOperator{\ord}{ord}
\DeclareMathOperator{\Pic}{Pic}
\DeclareMathOperator{\grad}{grad}

\numberwithin{equation}{section}

\title{The Locus of Plane Quartics with A Hyperflex}
\author{Xuntao Hu}
\address{Department of Mathematics, Stony Brook University, Stony Brook, New York, 11794}

\email{xuntao.hu@stonybrook.edu}
\date{}

\begin{document}

\maketitle
\begin{abstract}
Using the results of \cite{DPFSM}, we determine an explicit modular form defining the locus of plane quartics with a hyperflex among all plane quartics. As a result, we provide a direct way to compute the divisor class of the locus of plane quartics with a hyperflex within $\CM{3}$, first obtained in \cite{Cu89}. Moreover, the knowledge of such an explicit modular form also allows us to describe explicitly the boundary of the hyperflex locus in $\CM{3}$. As an example we show that the locus of banana curves (two irreducible components intersecting at two nodes) is contained in the closure of the hyperflex locus. We also identify an explicit modular form defining the locus of Clebsch quartics and use it to recompute the class of this divisor, first obtained in \cite{OS11}.
\end{abstract}

\section{Introduction}

We work over the field of complex numbers. A general line in $\PP^2$ intersects a plane quartic $C$ in four points. We call a line {\em bitangent} to $C$ if it intersects $C$ at two double points, denoted by $p$ and $q$. Thus $p+q=\frac{1}{2}K_C$ is an effective theta characteristic. The bitangent lines of a smooth plane quartic are in one-to-one correspondence with the gradients of theta functions with odd characteristics. The main interest of this paper lies in the case $p=q$, namely when a line intersects a plane quartic $C$ in a four-fold point. To fix notation, we call such a line a {\em hyperflex line}, and the intersection a {\em hyperflex point}. We call a smooth plane quartic that admits a hyperflex line a {\em hyperflex quartic}.

The locus $\HF$ of hyperflex quartics (in Teichm\"uller dynamics context also known as $\mathcal{H}^{odd}_3(4)$ or $\Omega\M{3}^{odd}(4)$) is a Cartier divisor (see \cite{Ver83} and \cite{Cu89}) in the moduli space of smooth genus three curves $\M{3}$. The class of its closure $[\overline{\HF}]$ in the Deligne-Mumford compactification $\CM{3}$ was computed in \cite{Cu89}:
$$
[\overline{\HF}]=308\lambda-32\delta_0-76\delta_1,
$$
where $\lambda$ is the Hodge class on $\overline{\mathcal{M}_3}$, and $\delta_0, \delta_1$ are the classes of the boundary divisors.

While the computation by Cukierman in \cite{Cu89} is algebraic, in this paper we determine explicitly the modular form defining the locus $\HF$, which gives a direct analytic approach to the problem.

We recall the notation in \cite{DPFSM}, in which a theta characteristic $(\ep,\delta)\in(\ZZ/2\ZZ)^3\times(\ZZ/2\ZZ)^3$ is denoted by $(i,j)$, where $i=4\ep_1+2\ep_2+\ep_3, j=4\delta_1+2\delta_2+\delta_3$. For instance, $([1,1,0],[0,1,1])$ is denoted by $(6,3)$.

Our main result is an explicit formula for a modular form $\Omega_{77}$, whose zero locus in $\A{3}(2)$ is the closure of the locus of plane quartics for which a fixed bitangent line is a hyperflex. Here $\A{3}(2)$ denotes the moduli space of principally polarized abelian varieties of dimension $3$ with a (full) level $2$ structure.
\begin{thm}
On $\A{3}(2)$, the modular form $\Omega_{77}(\tau)$ defined by:
\begin{equation*}
\begin{aligned}
\Omega_{77}(\tau):&= [\theta_{01}\theta_{10}\theta_{37}\theta_{43}\theta_{52}\theta_{75}\theta_{42}\theta_{06}\theta_{30}\theta_{21}\theta_{55}+\theta_{02}\theta_{25}\theta_{34}\theta_{40}\theta_{67}\theta_{76}\theta_{33}\theta_{05}\theta_{14}\theta_{60}\theta_{42}]^2
\\ &-4\theta_{01}\theta_{02}\theta_{10}\theta_{25}\theta_{34}\theta_{37}\theta_{40}\theta_{43}\theta_{52}\theta_{67}\theta_{75}\theta_{76}\theta_{00}\theta_{04}\theta_{57}\theta_{70}\theta_{61}\theta_{73}\theta_{20}\theta_{07}\theta_{00}\theta_{16}.
\end{aligned}
\end{equation*}
vanishes at the period matrix $\tau$ of a smooth plane quartic iff the bitangent line corresponding to $(i,j)=(7,7)$ is a hyperflex. Here $\theta_{ij}:=\theta_{ij}(\tau,0)$ is the Riemann theta constant with characteristics $(i,j)$.
\end{thm}

Once we know the modular form, computing the class of its zero locus is straightforward, and as a quick corollary we obtain (in Section $3$) a direct alternative proof of Cukierman's formula for $[\overline{\HF}]$. Moreover, by studying the Fourier-Jacobi expansion of the modular form one can explicitly describe the intersection of $\overline{\HF}$ with any boundary stratum of $\CM{3}$. As an example, we show (in Section 4) that the locus of ``banana" curves (two irreducible components intersecting at two nodes) is contained in the closure of the hyperflex locus.

The last section is a separate discussion of the locus of {\em Clebsch quartics}, the plane quartics that admit a polar pentagon. We determine an explicit modular form defining this locus, and apply the same method to confirm its divisor class in $\CM{3}$.

Besides the application to understanding the boundary of the hyperflex locus and the locus of Clebsch quartics, the modular forms are important for their own sake. Before this paper, the only known locus in $\M{3}$ that could be described using a modular form was the hyperelliptic locus. The corresponding modular form is the theta-null modular form --- the product of all even theta constants with characteristics --- which can be defined for arbitrary genus. There are other modular forms in genus higher than $3$ that describe loci of geometric significance, but in general they are rare. For a partial collection one can see \cite{CS12}.

{\small \section*{Acknowledgements}
Firstly I am very grateful to my advisor Samuel Grushevsky for his patience and suggestions throughout our discussions. I want to express my appreciation to the referee for all the valuable suggestions. I also thank David Stapleton for thoroughly proofreading this paper. Finally I am indebted to Riccardo Salvati Manni for his kind advice and encouragement.}

\section{Preliminaries}
\subsection{Theta Characteristics on a Plane Quartic}

We denote the {\em Siegel upper half-space} of dimension g by:
$$
\HH_g:=\{\tau\in \Mat(g\times g,\CC)\ |\ \tau=\tau^t,\Ima(\tau)>0\}.
$$
The moduli space of principally polarized abelian varieties (ppavs) of dimension $g$ $\A{g}=\Gamma_g\setminus\HH_g$ is the quotient of $\HH_g$ by the symplectic group $\Gamma_g:=\Sp(2g,\ZZ)$. We have the Torelli map $u:\M{g}\to\A{g}$, sending a curve to its Jacobian. Since our objects are plane quartics, our discussion will be in the case $g=3$. The Torelli map is dominant in this case, and can be extended to a morphism $\overline{u}:\CM{3}\to\CA{3}$, where $\CM{3}$ is the Deligne-Mumford compactification, and $\CA{3}$ is the perfect cone toroidal compactification, which in genus $3$ is the same as the second Voronoi, and the central cone toroidal compactifications.

For an abelian variety $A_{\tau}$, we denote the set of its two-torsion points by $A_{\tau}[2]\simeq(\ZZ/2\ZZ)^{2g}$, identifying a two-torsion point $m=(\tau\ep+\delta)/2$ with a characteristic $(\ep,\delta)\in(\ZZ/2\ZZ)^{2g}$.

\begin{definition}
The {\em Riemann theta function with characteristics $(\ep,\delta)$} is 
$$
\theta\begin{bsmallmatrix}\epsilon\\\delta\end{bsmallmatrix}(\tau,z):=\sum_{k\in\ZZ^g}\exp\Big[\pi i\big((k+\frac{\epsilon}{2})^t\tau(k+\frac{\epsilon}{2})+2(k+\frac{\epsilon}{2})^t(z+\frac{\delta}{2})\big)\Big].
$$
\end{definition}
When $(\ep,\delta)=(0,0)$, we have the usual Riemann theta function. For a fixed $\tau$, the theta function defines a section of a line bundle on the corresponding abelian variety $A_{\tau}$, which gives a principal polarization. 

 We define $e(m):=(-1)^{\ep\cdot\delta}=\pm 1$ to be the {\em parity} of $m$. The theta function with characteristics $(\ep, \delta)$ is an odd/even function of $z$ when $(\ep,\delta)$ is odd/even. Hence as a function of $\tau$, the {\em theta constant} $\theta\begin{bsmallmatrix}\epsilon\\\delta\end{bsmallmatrix}(\tau,0)$ is identically zero iff $(\ep,\delta)$ is odd, and $\grad_z\theta\begin{bsmallmatrix}\epsilon\\\delta\end{bsmallmatrix}(\tau,z)|_{z=0}$ vanishes identically iff $(\ep,\delta)$ is even.

In genus 3, the canonical image of a non-hyperelliptic curve is a plane quartic, and the bitangent lines to the plane quartic are given by the gradients of the theta functions with odd characteristics (see \cite[ch.~5]{Do12}).

\subsection{Modular Forms and the Level Covers of $\A{g}$}

\begin{definition}
Let $W$ be a finite dimensional complex vector space. Given an arithmetic subgroup $\Gamma\subset\Gamma_g$ and a representation $\rho:GL(g,\CC)\rightarrow GL(W)$, a holomorphic function $f:\HH_g\rightarrow W$ is called a {\em $\rho$-valued Siegel modular form} w.r.t $\Gamma$ if
$$
f(\gamma\circ\tau)=\rho(C\tau+D)\circ f(\tau)
$$
for any $\gamma=\begin{psmallmatrix}A&B\\C&D\end{psmallmatrix}\in\Gamma$, and any $\tau\in\HH_g$. For $g=1$ we also require $f$ to be regular at the cusps of $\Gamma\setminus\HH_1$.\end{definition}

If $W=\CC$, and $\rho(\gamma)=\det(C\tau+D)^k$, then the modular form is called a {\em weight $k$} (scalar) modular form for $\Gamma$.  We recall \cite{Igu72} the following transformation formula for theta functions with characteristics:
$$
\theta\begin{bsmallmatrix}\epsilon\\\delta\end{bsmallmatrix}(\gamma\tau,(C\tau+D)^{-1}z)=\phi\cdot\det(C\tau+D)^{1/2}\theta[\gamma\circ\begin{psmallmatrix}\epsilon\\\delta\end{psmallmatrix}](\tau,z)
$$
for any $\gamma=\begin{psmallmatrix}A&B\\C&D\end{psmallmatrix}\in\Gamma_g$ acting on the characteristics $(\ep,\delta)$ in the following way:
\begin{equation}\label{3}
\gamma\circ\begin{bmatrix}\epsilon\\\delta\end{bmatrix}=\begin{bmatrix}D&-C\\-B&A\end{bmatrix}\begin{bmatrix}\epsilon\\\delta\end{bmatrix}+\begin{bmatrix}\diag(C^tD)\\ \diag(A^tB)\end{bmatrix}.\end{equation}
In our case $\gamma\in\Gamma_g(4,8)\subset\Gamma_g$ (we will define it below), $\phi\equiv 1$ so we do not define $\phi$ in general. By differentiating with respect to $z_i$ we obtain:
$$
\frac{\partial}{\partial z_i}\theta\begin{bsmallmatrix}\epsilon\\\delta\end{bsmallmatrix}(\gamma\tau,(C\tau+D)^{-1}z)=\det(C\tau+D)^{1/2}\sum_{j}(C\tau+D)_{ij}\frac{\partial}{\partial z_j}\theta[\gamma\circ\begin{psmallmatrix}\epsilon\\\delta\end{psmallmatrix}](\tau,z)
$$
for any $\gamma\in\Gamma_g(4,8)$.

This is to say that the theta constants with characteristics are modular forms of weight $\frac{1}{2}$, and the gradients of the theta functions with characteristics evaluated at $z=0$ (see \cite{SM83}) are vector-valued modular forms for the representation $\rho=\det^{\frac{1}{2}}\otimes std$ with respect to a level subgroup $\Gamma_g(4,8)\subset\Gamma_g$, which is defined in general as follows:
$$
\Gamma_g(k):=\{\gamma\in\Gamma_g\ |\ \gamma\equiv\textbf{1}_{2g}\mod k\},
$$
$$
\Gamma_g(k,2k):=\{\gamma\in\Gamma_g(k)\ |\ \diag(C^tD)\equiv \diag(A^tB)\equiv 0\mod 2k\}.
$$
We will call the quotient $\A{g}(k):=\Gamma_g(k)\setminus\HH_g$, (resp. $\A{g}(k,2k):=\Gamma_g(k,2k)\setminus\HH_g$) the {\em moduli space of ppavs with a level $k$ structure} (resp. a level $(k,2k)$ structure). This cover of $\A{g}$ is Galois when $k$ is even.

\subsection{The Boundary of the Level Cover}

Recall that $\Pic_{\QQ}(\CA{3})=\QQ L\oplus\QQ D$, where $L$ is the first Chern class of the Hodge vector bundle $\mathbb{E}$ on $\A{3}$, and $D$ is the class of the boundary divisor (See \cite{HS02}). We further recall (see \cite{ACGII}) that $\Pic_{\QQ}(\CM{3})=\QQ\lambda\oplus\QQ\delta_0\oplus\QQ\delta_1$. Here $\lambda$ is the pullback of the class $L$ under the Torelli map, $\delta_0:=\bar{u}^*D$ is the class of the boundary component $\Delta_0$, the closure of the locus of irreducible curves with one node, and $\delta_1$ is the class of $\Delta_1$, the closure of the locus of nodal curves of compact type. The Torelli map contracts $\Delta_1$ onto the locus $P:=\A{1}\times\A{2}\subset\A{3}$.

By definition $\A{3}(2)$ is the moduli of ppavs together with a chosen symplectic basis for the group of two torsion points. There is a level toroidal compactification $\CA{3}(2)$ and the level map $\A{3}(2)\to\A{3}$ extends to a map $p:\CA{3}(2)\to\CA{3}$ of compactifications. The preimage $p^{-1}D$ is reducible, and its irreducible components are indexed by non-zero characteristics: $p^{-1}D=\cup_{n\in (\ZZ/2\ZZ)^{6}-0} D_n$. This enumeration of the components of $\Delta_0$ is also discussed in \cite{GrH12}.

The preimage $p^{-1}P$ is also reducible, and we now identify its irreducible components. For a generic point $[A]\in P$ we have $A=E\times A'$. The group of two-torsion points of $A$ splits as $A[2]\simeq(\ZZ/2\ZZ)^2\oplus(\ZZ/2\ZZ)^4$. Choosing such an isomorphism is the same as choosing a 2 dimensional symplectic subspace $V\subset(\ZZ/2\ZZ)^6$. Hence the irreducible components of $p^{-1}P$ are labeled by the choice of such subspaces, and we denote them by $P_V$.

Throughout the paper we will use the following fibre product diagram:
\begin{equation}\label{n}
\begin{CD}
\CM{3}(2) @>\bar{u}'>> \CA{3}(2)\\
@VVp'V @VVpV\\
\CM{3} @>\bar{u}>> \CA{3}
\end{CD}
\end{equation}
In this diagram $\CM{3}(2)$ is the pullback of $\CA{3}(2)$ through $\bar{u}$, which parametrizes stable genus three curves with a chosen symplectic basis for the group of two torsion points in their Jacobians. The maps $p$ and $p'$ are forgetful, while the maps $\bar{u}$ and $\bar{u}'$ are extended Torelli maps. This diagram will be used in the computation of the class in $\CM{3}$ defined by the pullback of the modular form, which is computed on $\CA{3}(2)$.

\section{The Hyperflex Locus}

The {\em hyperflex locus} $\HF$ is defined to be the subset of $\M{3}$ consisting of plane quartics with at least one hyperflex point. It can be shown that $\HF$ is an irreducible divisor:

\begin{prop}[{\cite[Ch.~1,~Prop.~4.9]{Ver83}}]\label{j}
$\HF$ is an irreducible, five-dimensional subvariety of $\M{3}$, and it is closed in $\M{3}-H_3$ where $H_3$ is the hyperelliptic locus.
\end{prop}

We denote the closure of $u(\HF)$ in $\A{3}$ by $HF$. We define $HF_m\subset\A{3}(2)$ to be the set of ppavs $(J(C),i)$ where the bitangent line corresponding to the odd characteristic $m$ under the basis defined by $i:J(C)[2]\simeq(\ZZ/2\ZZ)^6$ is a hyperflex line to $C$.

To determine the scalar modular form w.r.t.~$\Gamma_3(2)$, whose zero locus is equal to $HF_{77}$, we need to know the equation of a plane quartic in terms of its bitangents. Such a formula was known classically for an individual curve (see \cite[Ch.~5]{Do12}). Only recently Dalla Piazza, Fiorentino and Salvati Manni obtained such an expression globally over the moduli space of genus $3$ curves \cite{DPFSM}. They derived an eight by eight symmetric matrix parametrizing the bitangents of a given plane quartic, such that the determinant of any four by four minor of the matrix gives the equation of the quartic. We recall their notation and results.

\begin{definition}\label{e}
\begin{enumerate}
\item A triple of characteristics $m_1,m_2,m_3$ is called {\em azygetic} (resp. {\em syzygetic}) if
$$
e(m_1,m_2,m_3)=e(m_1)e(m_2)e(m_3)e(m_1+m_2+m_3)=-1\ \text{(resp. }1).
$$
\item A $(2g+2)$-tuple of characteristics is called a {\em fundamental system} if any subset of three elements is azygetic.
\end{enumerate}
\end{definition}

For a more detailed discussion see \cite[Chp.~5]{Do12}. In our case $g=3$, any fundamental system consists of 8 characteristics, within which either $3$ or $7$ elements are odd.

We define
$$
b_{ij}:=\grad_z\theta\begin{bsmallmatrix}\epsilon\\\delta\end{bsmallmatrix}(\tau,z)|_{z=0},
$$
where $i=4\epsilon_1+2\epsilon_2+\epsilon_3, j=4\delta_1+2\delta_2+\delta_3$, and denote the Jacobian determinant by:
$$
D(n_1,n_2,n_3):=b_{n_1}\wedge b_{n_2}\wedge b_{n_3}.
$$
It is known that $D$ is a scalar modular form of weight $\frac{5}{2}$ that can be written in terms of theta constants using Jacobi's derivative formula:

\begin{prop}[\cite{Igu81}]{\label{b}}
Let $n_1, n_2, n_3$ be an azygetic triple of odd theta characteristics, then there exists a unique quintuple of even theta characteristics $m_1$, $m_2$, $m_3$, $m_4$, $m_5$ such that the 8-tuple forms a fundamental system. We have
$$
D(n_1,n_2,n_3)=\pm\pi^3\cdot\theta_{m_1}\theta_{m_2}\theta_{m_3}\theta_{m_4}\theta_{m_5}.
$$
\end{prop}

The following proposition is the result of Dalla Piazza, Fiorentino, Salvati Manni.
\begin{prop}[{\cite[Cor.~6.3]{DPFSM}}]\label{k}
Let $\tau$ be the period matrix of the Jacobian of a plane quartic $C$, then the equation of the image of $C$ under canonical embedding is given by the determinant of the following symmetric matrix:
$$
Q(\tau,z):=\begin{pmatrix}
0&\frac{D(31,13,26)}{D(77,31,26)}b_{77}&\frac{D(22,13,35)}{D(77,31,26)}b_{64}&\frac{D(77,64,46)}{D(77,31,26)}b_{51}\\
*&0&\frac{D(22,13,35)}{D(77,46,51)}b_{13}&\frac{D(77,13,31)}{D(77,31,26)}b_{26}\\
*&*&0&\frac{D(64,13,22)}{D(77,31,26)}b_{35}\\
*&*&*&0\\
\end{pmatrix}.
$$
\end{prop}

Note that $\{b_{ij}\}$ are linear expressions in the coordinates of $\mathbb{P}H^0(C,K_C)^\vee$, and the determinant is a quartic polynomial in the $b_{ij}$, with coefficients being rational functions in Jacobian determinants. Using this we derive the modular form $\Omega_{77}$:
\begin{thm}{\label{a}}
Let $\Omega_{77}$ be the following modular form with respect to $\Gamma_3(2)$:
\begin{multline}\label{4}
\Omega_{77}:= \big[\theta_{01}\theta_{10}\theta_{37}\theta_{43}\theta_{52}\theta_{75}\cdot D(77,64,13)+\theta_{02}\theta_{25}\theta_{34}\theta_{40}\theta_{67}\theta_{76}\cdot D(77,51,26)\big]^2
\\-4\theta_{01}\theta_{02}\theta_{10}\theta_{25}\theta_{34}\theta_{37}\theta_{40}\theta_{43}\theta_{52}\theta_{67}\theta_{75}\theta_{76}\cdot D(77,64,51)\cdot D(77,13,26).
\end{multline}
Then its zero locus in $\A{3}(2)$ is equal to $HF_{77}$.
\end{thm}

Using the Jacobi's derivative formula given by Proposition \ref{b}, $\Omega_{77}$ can be rewritten as a polynomial in theta constants, which (up to  the constant factor of $\pi^6$) is equal to the modular form in Theorem 0.1. Thus proving Theorem~\ref{a} will complete the proof of our main result.

The proof is by directly computing the bitangents, and uses the following lemma:

\begin{lem}\label{f}
Let $l=l_1x+l_2y+l_3z$ be the equation of a line in $\mathbb{P}^2$ in homogeneous coordinates $(x:y:z)$, and suppose $m, n, k, s$  are lines written similarly. Then the two intersection points of the line $l=0$ and the quadric $mk-ns=0$ coincide if and only if the following expression vanishes:
\begin{equation}\label{5}
\Psi_{l,m,n,k,s}=\Big(\begin{vsmallmatrix}
l_1&l_2&l_3\\m_1&m_2&m_3\\k_1&k_2&k_3
\end{vsmallmatrix}+\begin{vsmallmatrix}
l_1&l_2&l_3\\n_1&n_2&n_3\\s_1&s_2&s_3
\end{vsmallmatrix}\Big)^2
-4\cdot\begin{vsmallmatrix}
l_1&l_2&l_3\\m_1&m_2&m_3\\n_1&n_2&n_3
\end{vsmallmatrix}\cdot\begin{vsmallmatrix}
l_1&l_2&l_3\\k_1&k_2&k_3\\s_1&s_2&s_3
\end{vsmallmatrix}.
\end{equation}
\end{lem}
\begin{proof}
The proof is a direct computation: we plug in the equation of $l$ into $\{mk-ns=0\}$ and get:
\begin{align*}
&\big[(m_1l_2-m_2l_1)x+(m_3l_2-m_2l_3)z\big]\cdot \big[(k_1l_2-k_2l_1)x+(k_3l_2-k_2l_3)z\big]\\
&-\big[(n_1l_2-n_2l_1)x+(n_3l_2-n_2l_3)z\big]\cdot \big[(s_1l_2-s_2l_1)x+(s_3l_2-s_2l_3)z\big]=0.
\end{align*}
We dehomogenize at $z$. The discriminant $F$ of this quadric (in $x$) is a homogenous polynomial of degree 8 in the coefficients of $l,m,n,k,s$. We further observe that $F$ is divisible by $l_2^2$. Define $\Psi:={F}/{l_2^2}$. One can verify that $\Psi$ is independent of the dehomogenization.
\end{proof}

\begin{proof}[Proof of Theorem \ref{a}]
Using Proposition~\ref{b}, we write the coefficients of $Q(\tau,z)$ given by Proposition \ref{k} as rational functions of even theta constants. By clearing the denominators we have the equation of the plane quartic:
\begin{equation}\label{p}
\det Q(\tau,0)=(\theta_{75}\theta_{52}\theta_{43})^4\cdot(\theta_{04}^2\theta_{73}\theta_{60})^2\cdot[(af)^2+(be-cd)^2-2(af)(be+cd)]=0,
\end{equation}
where
$$
  \begin{array}{c}
   a=\theta_{66}\theta_{41}\theta_{50}b_{77},\qquad
   b=\theta_{70}\theta_{52}\theta_{43}b_{64},\qquad
   c=\theta_{40}\theta_{76}\theta_{67}b_{51}, \\
   d=\theta_{02}\theta_{25}\theta_{34}b_{13}, \qquad
   e=\theta_{37}\theta_{01}\theta_{10}b_{26}, \qquad
   f=\theta_{24}\theta_{12}\theta_{03}b_{35}.
  \end{array}
$$

Recall that on $\A{3}$ the vanishing of the theta-null modular form (the product of all even theta constants) defines the hyperelliptic locus, which we know is disjoint from the hyperflex locus by Proposition \ref{j}. More concretely, the vanishing locus of any even theta constant is an irreducible component of the hyperelliptic locus. Thus the vanishing locus defined by the common factor $(\theta_{75}\theta_{52}\theta_{43})^4\cdot(\theta_{04}^2\theta_{73}\theta_{60})^2$ lies in the hyperelliptic locus, we hence eliminate this common factor from $\det Q(\tau,0)$. Rewrite the remaining part:
$$
(af)^2+(be-cd)^2-2(af)(be+cd)=a\cdot F+(be-cd)^2
$$ where $F$ is a homogenous degree 3 polynomial in $a,b,c,d,e,f$. This is an equation of the quartic to which $\{a=0\}$ is a bitangent line. And the two tangent points are given by the two intersections of $\{a=0\}$ and the double conic $\{(be-cd)^2=0\}$:
$$
\{a=0\}\cap \{a\cdot F+(be-cd)^2=0\}=\{a=0\}\cap 2\cdot \{be-cd=0\}.
$$

By the lemma, plugging $a,b,c,d,e$ in to (\ref{5}) we have $$\Psi_{a,b,c,d,e}=\theta_{66}\theta_{73}\theta_{41}\theta_{50}\theta_{04}\cdot\Omega_{77}$$ where $\Omega_{77}$ is defined in~\eqref{4}. Finally we eliminate the common factor $\theta_{66}\theta_{73}\theta_{41}\theta_{50}\theta_{04}$ as its vanishing locus also lies in the hyperelliptic locus, thus $\Omega_{77}$ is the correct modular form.
\end{proof}

Using the modular form we can now compute the class of the hyperflex locus $\HF$ in $\M{3}$:
\begin{cor}
The class $[\HF]\in H^2(\M{3},\QQ)$ is equal to $308\cdot\lambda$.
\end{cor}
\begin{proof}
First we need to compute the weight of the modular form $\Omega_{77}$. The weight of $D(n_1,n_2,n_3)$ is $\frac{5}{2}$ and the weight of each $\theta_m$ is $\frac{1}{2}$. Therefore $12\cdot\frac{1}{2}+2\cdot\frac{5}{2}=11$ is the weight of the scalar modular form $\Omega_{77}$ with respect to $\Gamma_3(2)$.

Set-theoretically the hyperflex locus $HF\subset\A{3}$ is the image of $HF_{77}\subset\A{3}(2)$ under the level two covering map $p$. Moreover, for any odd characteristic $m$ we have $p(HF_m)=HF$. By the computation of weight we have $[HF_{77}]=11\cdot p^*L\in H^2(\A{3}(2),\QQ)$. We have
\begin{equation}\label{i}
p^*[HF]=\sum_{m\text{ odd}}[HF_m]=28\cdot 11\cdot p^*L=308\cdot p^*L.
\end{equation}
The second equality is due to the fact that $p$ is a Galois covering. Hence for all odd $m$, the class of $HF_m$ is equal to that of $HF_{77}$. Taking the pushforward of (\ref{i}) by $p$, by projection formula we obtain:
$$
[HF]=308\cdot L.
$$
Pulling this back under the Torelli map $u$, we obtain the result.
\end{proof}

\section{Extension of Theta Constants and Theta Gradients to the Boundary}

In order to use a modular form to compute the class of the closure of its zero locus in the compactification of $\M{3}$, we need to know its vanishing order at the boundary. We will first compute the extension of theta constants and theta gradients to the boundary.

\subsection{Characterization of the Orbits of the $\Gamma_g$-action on Sets of Characteristics}

We recall the following standard definition (see \cite{Do12} for a more detailed discussion):
\begin{definition}
A sequence of characteristics $m_1,m_2,\ldots,m_s$ is called {\em essentially independent} if for any choice of $1\leq i_1<i_2<\ldots<i_{2k}\leq s$ with $k\geq1$ we have
$$
m_{i_1}+m_{i_2}+\ldots+m_{i_{2k}}\neq 0 \mod 2.
$$
\end{definition}

Recall the notation $D_n$ and $P_V$ for the irreducible components of $p^{-1}D$ and $p^{-1}P$ in $\CA{3}(2)$. For the purpose of computing the vanishing orders of $\theta_m$ and $\grad_z\theta_m$, we need the following characterization of the orbits of the $\Gamma_g$-action on the sets of characteristics (recall that the action is defined by (\ref{3})).

\begin{prop}[\cite{Igu72}, \cite{SM94}] \label{y}
Two ordered sequences $m_1,m_2,\ldots,m_r$ and $n_1,n_2,\ldots,n_r$ of characteristics are conjugate under the action of $\Gamma_g$ if and only if $e(m_i)=e(n_i)$, and $e(m_i,m_j,m_k)=e(n_i,n_j,n_k)$ for any $1\leq i\leq r,1\leq i<j<k\leq r$, and if furthermore for any $I\subset\{1,2,\ldots, r\}$, $\{m_i\}_{i\in I}$ is an essentially independent subsequence if and only if $\{n_i\}_{i\in I}$ is an  essentially independent subsequence.
\end{prop}

Note that if $(m,n)$ and $(m',n')$ lie in the same $\Gamma_g$-orbit, then on the level two cover $\CA{3}(2)$, we must have $\ord_{D_n}\theta_m=\ord_{D_{n'}}\theta_{m'}$. Thus it suffices to compute this vanishing order for one element in each $\Gamma_g$-orbit of pairs $(m,n)$.

Since the group $\Gamma_g$ acts transitively on the set $D_n$ of boundary components, each orbit of $(m,n)$ under $\Gamma_g$ contains all possible values of $n$. We thus fix the boundary component $D_n$, and apply the proposition to find the orbits of $(m,n)$ when $m$ is varying: consider the set of triples $(m,n,0)$ where $n$ is fixed and $m$ is even (resp. odd), so that the parity of $m$ and $n$ remains the same. Thus by Proposition \ref{y}, the orbits only depend on $e(m,n,0)$. By definition $e(m,n,0)=e(m)e(n)e(m+n)$, hence there are two orbits of pairs $(m,n)$ for $n$ fixed, distinguished by the parity of $m+n$.

In order to calculate the vanishing order of theta constants on $P_V$,
we will also need the description of the orbits of the $\Gamma_g$-action on the set of pairs $(m, V)$, where $V$ is a symplectic 2-dimensional subspace of $(\ZZ/2\ZZ)^{2g}$.

\begin{prop}\label{z}
Let $V=span(n_1,n_2)$ be a fixed symplectic 2-dimensional subspace of $(\ZZ/2\ZZ)^{2g}$. There are two $\Gamma_g$-orbits of pairs $(m,V)$. These possibilities are distinguished by the number of even elements in the set $\{m+n_1,m+n_2,m+n_1+n_2\}$ being $1$ or $3$.
\end{prop}
\begin{proof}
Let $X$ be the set of pairs $(m,V)$, $Y$ be the set of quadruples $\{m, n_1,n_2,n_1+n_2\}$. Let the map $q:Y\to X$ be the quotient under the symmetric group $S_3$ permuting the last three elements. Thus $q$ is $\Gamma_g$-equivariant. Denote the induced map by $q': Y/\Gamma_g\rightarrow X/\Gamma_g$.

By Proposition \ref{y}, the $\Gamma_g$-action on $Y$ has eight orbits only depending on the parities of the triple $\{m+n_1,m+n_2,m+n_1+n_2\}$, namely $Y/\Gamma_g\simeq\mathbb{F}_2^3$. The map $q'$ forgets the order of elements in the triple. Hence the orbits of $\sigma$ depend only on the number of even elements in the triple $\{m+n_1,m+n_2,m+n_1+n_2\}$.

Let $\omega$ is the standard symplectic form. Observe that for $m$ odd and $n_1,n_2$ satisfying $\omega(n_1,n_2)\neq0$, we have $e(m+n_1+n_2)=e(m+n_1)e(m+n_2)$. The only possibilities for the number of even elements in the triple $\{m+n_1,m+n_2,m+n_1+n_2\}$ are thus 1 or 3. 
\end{proof}

\subsection{Extension to the Boundary}

Extension of theta constants and theta gradients to the boundary component $D_n$ is done in \cite{GrH12}. The vanishing orders are computed using the Fourier-Jacobi expansion of the theta function (this expansion is convenient for the computation that we will later do on $\Delta_1$):

\begin{equation}{\label{1}}
\theta\begin{bsmallmatrix}
\epsilon'&\epsilon\\
\delta'&\delta\\
\end{bsmallmatrix}\big(\begin{bsmallmatrix}
\tau'&b\\
b^t&\tau\\
\end{bsmallmatrix},0\big)=\sum_{k'\in\mathbb{Z},k''\in\mathbb{Z}^{g-1}}\exp \pi i\big[2(k'+\frac{\epsilon'}{2})b(k''+\frac{\epsilon}{2})\big]A(k',k'')
\end{equation}
where
$$
A(k',k'')=\exp \pi i\left([(k'+\frac{\epsilon'}{2})^2\tau'+(k'+\frac{\epsilon'}{2})\delta]+[(k''+\frac{\epsilon}{2})^t\tau(k''+\frac{\epsilon'}{2})+(k''+\frac{\epsilon}{2})^t\delta]\right).
$$

By the characterization of the orbits of the $\Gamma_g$-action we only need to work on a chosen boundary component $D_{n_0}$ corresponding to $n_0=\begin{bsmallmatrix}0&0&...&0\\1&0&...&0\end{bsmallmatrix}$. Due to the parity of the theta constants and the theta gradients, we assume the characteristic $m$ is even for $\theta_m$, and is odd for $\grad_z\theta_m$, so that they don't vanish identically.

The vanishing orders of $\theta_m(\tau,0)$ and $\grad_z\theta_m(\tau,0)$ on $D_{n_0}$ are as follows:

\begin{prop}\label{h} \cite[Prop.~3.3]{GrH12}
We have the following:
\begin{equation}
\\ord_{D_{n_0}} \theta_m(\tau, 0)=\begin{cases}
   0 & \mbox{if } e(m+n_0)=1 \\
   \frac{1}{8}     & \mbox{if }e(m+n_0)=-1
  \end{cases}
\end{equation}

\begin{equation}
\ord_{D_{n_0}}\grad_z\theta_m(\tau,z)|_{z=0}=\begin{cases}
  (\frac{1}{2},0,\ldots, 0) & \mbox{if } e(m+n_0)=-1 \\
  (\frac{1}{8},\frac{1}{8},\ldots,\frac{1}{8})    & \mbox{if }e(m+n_0)=1
  \end{cases}
  \end{equation}
The notation above indicates the vanishing order for each partial derivative \\
$(\frac{\partial}{\partial z_1}\theta,\frac{\partial}{\partial z_2}\theta\ldots\frac{\partial}{\partial z_g}\theta)$.
\end{prop}

For the boundary $\Delta_1$, we can do a similar computation, which to our knowledge has not been done in literature. Following \cite{Yam80} and \cite{Fay73}, we will consider the pinching/plumbing family of Riemann surfaces pinching a cycle homologous to zero. For a Riemann surface $C$ of genus $g$, we fix an element of $\pi_1(C)$ which maps to zero in homology and is represented by a simple closed curve, and consider the plumbing family $\mathcal{C}\subset\overline{\mathcal{M}}_3$ parameterized by shrinking the length $s$ of this curve to zero: for $s\ne 0$ the curve $C_s$ is smooth, while for $s=0$ the curve $C_0$ lies in $\Delta_1$. We denote the period matrix of $C_s$ by $\tau_s$. By \cite[cor.~2]{Yam80}, the expansion of $\tau_s$ near $s=0$ is:
$$
\tau_s=\begin{bmatrix}
\tau_1&0\\
0&\tau_2\\
\end{bmatrix}-s\begin{bmatrix}
0&R\\
R^T&0\\
\end{bmatrix}+O(s),
$$
where $\tau_1\in \Mat_{g_1\times g_1}(\CC)$ and $\tau_2\in \Mat_{g_2\times g_2}(\CC)$ where $g_1$ and $g_2$ are the genera of the two irreducible components of $C_0$, and $R\in\Mat_{g_1\times g_2}(\CC)$ is some matrix independent of $s$. We recall the factorization
\begin{equation}\label{2}
\theta\begin{bsmallmatrix}
\epsilon'&\epsilon\\
\delta'&\delta\\
\end{bsmallmatrix}(\begin{bsmallmatrix}
\tau'&0\\
0&\tau''\\
\end{bsmallmatrix},0)=\theta\begin{bsmallmatrix}\epsilon'\\\delta'\end{bsmallmatrix}(\tau',0)\times\theta\begin{bsmallmatrix}\epsilon\\\delta\end{bsmallmatrix}(\tau'',0).
\end{equation}
In our case $g_1=1$, $g_2=2$. 
As in the case of $\Delta_0$, in the following discussion we assume $m=\begin{bsmallmatrix}\epsilon'&\epsilon\\ \delta'&\delta\\ \end{bsmallmatrix}$ is even for $\theta_m$, odd for $\grad_z\theta_m$. Since $\ep'\delta'+\ep\delta=0$, the product in (\ref{2}) vanishes if and only if $\ep'\cdot\delta'=1$ (because then both of the factors are odd functions with respect to $z$). We now substitute $\tau_s$ given by the family above into the Taylor expansion with respect to $b=s\cdot R$ given by~\eqref{1}, which yields $\ord_b\theta_m(\tau,0)=1$ if $\ep'\cdot\delta'=1$, and it does not vanish generically otherwise.

Take now the component $P_{V_0}$ corresponding to:

$$V_0=Span(n_1=\begin{bsmallmatrix}1&0&0\\0&0&0\end{bsmallmatrix},n_2=\begin{bsmallmatrix}0&0&0\\1&0&0\end{bsmallmatrix}).$$
Then $\bar{u}^{-1}(P_{V_0})$ is a component of $p'^{-1}\Delta_1$ in $\CM{3}(2)$. Thus from the discussion above, one can conclude:

\begin{prop} On the boundary component $\bar{u}^{-1}P_{V_0}$ in $\CM{3}(2)$, we have:
\begin{equation}
\ord_{b}\theta_m(\tau,0)=\begin{cases}
   1 & \mbox{if } e(m+n_1)=e(m+n_2)=-1\\
   0    & \mbox{otherwise }
  \end{cases}
\end{equation}
\begin{equation}
\ord_{b}\grad_z\theta_m(\tau,0)=\begin{cases}
   (0,1,1) & \mbox{if } e(m+n_1)=e(m+n_2)=1\\
   (1,0,0)    & \mbox{otherwise. }
  \end{cases}
\end{equation}
The notation again indicates the vanishing order for each partial derivative of $\theta$.
\end{prop}
\begin{proof}
We have the observation:
\begin{align*}
e(m+n_1)&=(-1)^{(\epsilon'+1)\delta'+\epsilon^t\delta}=(-1)^{\delta'}\cdot e(m)\\
e(m+n_2)&=(-1)^{\epsilon'}\cdot e(m).
\end{align*}
Since we assume $m$ is even, we have $e(m)=1$. The conditions in the proposition are the same as $\ep'=\delta'=1$. The computation for theta gradients is parallel to the computation for theta constants. We therefore omit it here.
\end{proof}

\subsection{Class of the Closure of the Hyperflex Locus}

Let $\Omega_m$ be the image of $\Omega_{77}$ under the action of $\Gamma_g$, so that it is a modular form with respect to $\Gamma_3(2)$ whose zero locus in $\A{3}(2)$ is $HF_m$.  Denote for simplicity $d_{m,n}:=\ord_{D_n}\Omega_m(\tau,0)$, and by $p_{m,V}$  the vanishing order of the pull-back of $\Omega_m(\tau,0)$ on $\bar{u}^{-1}P_{V}$. There are only two possible values of $d_{m,n}$ corresponding to the two $\Gamma_g$-orbits on $(m,n)$ --- we denote these vanishing orders by $d_0$ and $d_1$ for the cases $e(m+n)=0$ and $1$. Similarly let $p_1$ and $p_3$ be the values of $p_{m,V}$ in the $\Gamma_g$-orbit on the set of pairs $(m,V)$ where the subindex denotes the number of even elements in the triple from Proposition \ref{z}. We have the following:

\begin{prop}
In $\CM{3}$, we have
$$
[\overline{\HF}]=308\cdot\lambda-(16d_0+12d_1)\cdot\delta_0-(10p_3+18p_1)\cdot\delta_1.
$$
\end{prop}

\begin{proof}
It can be concluded from a direct computation that for each $n\in (\ZZ/2\ZZ)^6-0$, there are 16 odd $m$ such that $m+n$ is even, and 12 odd $m$ such that $m+n$ is odd; for a fixed $V$, there are 18 odd theta characteristics $m$ lying in the orbit corresponding to the case when the number of even elements in the triple $(m+n_1,m+n_2,m+n_1+n_2)$ is 1, and 10 odd theta characteristics in the other orbit.

Consider the commutative diagram~(\ref{n}). Summing over all $m$, on $\CM{3}(2)$ we have:
$$
\bar{u}'^*\left(\sum_{m \text{ odd}}{[\CHF_m]}\right)=308\cdot p'^*\lambda-\sum_{m,n}{d_{m,n}\cdot\bar{u}'^*D_n}-\sum_{V,n}{p_{m,V}\cdot\bar{u}'^*P_V}.
$$
On the right hand side we have:
\begin{align*}
\sum_{m,n}{d_{mn}\cdot\bar{u}'^*D_n} &= \bar{u}'^*\left(\sum_{m+n \mbox{ even} }d_0D_n+\sum_{m+n\mbox{ odd}}d_1D_n\right)\\
&=\bar{u}'^*\left(d_0\sum_{n}16D_n+d_1\sum_{n}12D_n\right)\\
&=\bar{u}'^*\left((16d_0+12d_1)\sum_{n}D_n\right)\\
&=(16d_0+12d_1)\cdot \bar{u}'^*(p^*D)\\
&=(16d_0+12d_1)\cdot p'^*{\delta_0}.
\end{align*}
Similarly we have $\sum_{V,n}{p_{m,V}\cdot\bar{u}'^*P_V}=(10p_3+18p_1)\cdot p'^*{\delta_1}$. For the same reason as in Equation (\ref{i}), we have $\bar{u}'^*\left(\sum_{m \text{ odd}}{[\CHF_m]}\right)=p'^*[\overline{\HF}]$. Pushing forward by $p'$, by the projection formula both sides are multiples of $\deg(p')$. Note that the level cover map branches along the boundary components, but the projection formula applies regardless of the branching. Finally we divide both sides by $\deg(p')$ and  have the equality claimed.
\end{proof}

We now use the results from previous subsection to compute $d_0,d_1,$ and $p_1,p_3$.

\begin{prop}\label{g}
We have the following:
\begin{equation}
d_{m,n}=\begin{cases}
   \frac{5}{4} & \text{if $m+n$ is even} \\
   1    & \text{otherwise}
  \end{cases}
\end{equation}
\begin{equation}
p_{m,V}=\begin{cases}
   4 & \text{all elements in the triple are even} \\
   2  & \text{otherwise}.
  \end{cases}
\end{equation}
\end{prop}

\begin{proof}
To do the calculation it is enough to choose a special representative in each orbit. Fix $m=77$. For $d_0$ we choose $n=04$ so that $m+n$ is even. The vanishing orders on $D_{04}$ of $\theta_{43}, \theta_{52}, \theta_{75}, \theta_{40},\theta_{67}, \theta_{76}$ are all $1/8$, while other theta constants involved in the expression of $\Omega_{77}$ do not vanish identically on $D_{04}$. We also have $\ord_{D_{04}}D(77,64,13)=\ord_{D_{04}}D(77,51,26)=1/4, \ord_{D_{04}}D(77,64,51)=3/8$, and $\ord_{D_{04}}D(77,13,26)=1/8$. Hence we have $d_0=\min\{(3/8+1/4)\times 2,6/8+3/8+1/8\}=5/4$.

Similarly we choose $n=06$ for the case $m+n$ is odd. The vanishing orders on $D_{06}$ of $\theta_{43}$, $\theta_{52}$, $\theta_{37}$, $\theta_{40}$, $\theta_{25}$, $\theta_{34}$ are all $1/8$, and all other theta constants in $\Omega_{77}$ do not vanish identically on $D{06}$. We also have $\ord_{D_{06}}D(77,64,13)=1/2$, $\ord_{D_{06}}D(77,51,26)=1/4$, and $\ord_{D_{06}}D(77,64,51)=\ord_{D_{06}}D(77,13,26)=1/8$, hence $d_1=\min\{5/4,1\}=1$.

To compute the vanishing orders on $P_V$, we now choose the standard symplectic 2-dim subgroup $V_0$ as in section 3.2. In this case $m+n_1,m+n_2$ are both even, and we can thus compute $p_3$. We will have
$\ord_{V_0}D(77,64,13)=\ord_{V_0}D(77,64,51)=1$, and $
\ord_{V_0}\theta_{75}=\ord_{V_0}\theta_{67}=\ord_{V_0}\theta_{76}=1$ and all the others are zero, hence $p_3=\min\{(1+1)\times 2,4\}=4$.

Similarly we choose $V_1$ generated by $n_1=[101,000],n_2=[000,100]$ to compute $p_1$. We have $\ord_{V_1}D(77,64,51)=1$, $\ord_{V_1}\theta_{43}=\ord_{V_1}\theta_{76}=1$, all others are non-vanishing. We hence have $p_1=\min\{1\cdot 2,1+1+1\}=2$.

Lastly, since the expression of the modular form is explicit, one checks by hand that the lowest order term in each case does not get cancelled.
\end{proof}

Combining Proposition \ref{y} and Proposition \ref{z}, we can verify Cukierman's result in \cite{Cu89}:
\begin{cor}
In $\overline{\mathcal{M}_3}$, we have
$$
[\overline{\HF}]=308\cdot\lambda-32\cdot\delta_0-76\cdot\delta_1.
$$
Also, the class $[\overline{HF}]$ in $\CA{3}$ is equal to $308\cdot L-32\cdot D$.
\end{cor}
\begin{proof}
We only need to plug in the values $d_0=5/4,d_1=1,p_1=2,p_3=4$ in Proposition 3.6. And the second claim follows easily.
\end{proof}

\section{Boundary Strata of Higher Codimension}

Using the modular form $\Omega_{77}$, we can apply similar arguments to find the intersection of any boundary components of $\CM{3}$ with the closure of the hyperflex locus $\overline{\HF}$. As an application, we consider the boundary stratum $T\subset\CM{3}$ parameterizing stable curves consisting of two genus one curves intersecting at two nodes (so-called ``banana curves"). This boundary stratum is contained in $\Delta_0$ and is indeed an irreducible component of the self-intersection of $\Delta_0$. These curves are interesting examples of stable curves of non-pseudocompact type.

\begin{prop}\label{m}
The boundary locus $T$ is contained in the hyperflex locus $\overline{\HF}$.
\end{prop}

\begin{rmk}
This result was recently also shown by a different approach in \cite{Che15}.
\end{rmk}

To prove the proposition, we recall the general variational formula for the degeneration of the period matrix of a Riemann surface of genus $g$ with $n$ nodes. For $n=1$, see \cite{Yam80} and \cite{Fay73}, and for $n\geq 1$, see \cite{Tan89} and also \cite{Tan91}. For $i=1\ldots n$, fix elements  $[S_i]\in\pi_1(C)$ represented by simple closed curves $S_i$ with lengths $0\leq s_i\ll 1$. We also fix a homology basis $\{A_j, B_j\}_{j=1}^{g}$ such that for $1\leq i\leq n$, $S_i$ is homotopic to one of the $A_j$ (possibly with a sign).

\begin{lem}[{\cite[Thm 5]{Tan89}}]
For any $1\leq h,k\leq g$, the function
$$
f_{h,k}(s_1,\dots,s_n):=\exp\big(2\pi i\tau_{h,k}(s_1,\dots,s_n)\big)\cdot\prod_{i=1}^n s_i^{-N_{i,h}\cdot N_{i,k}}$$
is holomorphic in $0\leq s_i\ll 1$ for $i=1\ldots n$, where $N_{i,j}$ is the intersection product of $S_i$ and $B_j$, and $\big[\tau_{h,k}(s_1,\ldots, s_n)\big]_{g\times g}$ is the period matrix for $C(s_1,\ldots,s_n)$.
\end{lem}

For the boundary locus $T$, we have $g=3$ and $n=2$ in the above lemma. Furthermore, we choose the homology basis to be the standard one with intersection matrix $Id$, so $S_1$ and $S_2$ are both homotopic to $A_1$. By the lemma above, we have

\begin{equation}\label{l}
2\pi i\tau_{h,k}=\begin{cases}
\ln s_1+\ln s_2+f_{h,k}(s_1,s_2)&\text{for }(h,k)=(1,1)\\
f_{h,k}(s_1,s_2) & \text{otherwise.}
  \end{cases}
\end{equation}

We thus write $\tau=\begin{bsmallmatrix}
\tau_1&b_1&b_2\\
b_1&\tau_2&c\\
b_2&c&\tau_3\\
\end{bsmallmatrix}$, and recall Fourier-Jacobi expansion from (\ref{1}). We can then deduce the following facts about $\theta\begin{bsmallmatrix}
\ep\\
\delta\\
\end{bsmallmatrix}(\tau,0)$, which we will use in the proof of Proposition \ref{m}:
\begin{enumerate}

\item If $\ep_1=1$, then $\theta\begin{bsmallmatrix}\ep\\\delta\\\end{bsmallmatrix}(\tau,0)=\exp(\frac{1}{4}\pi i\tau_1)\cdot\exp(2\pi i\delta_1)\cdot\theta\begin{bsmallmatrix}\ep_2&\ep_3\\\delta_2&\delta_3\\\end{bsmallmatrix}\Big(\begin{bsmallmatrix}\tau_2&c\\c&\tau_3\end{bsmallmatrix},(\frac{b_1}{2},\frac{b_2}{2})\Big)+O(s_1)+O(s_2).$ Note that  due to (\ref{l}), we have $\exp(\pi i\tau_1)=s_1^{\frac{1}{2}}s_2^{\frac{1}{2}}\cdot \exp G(s_1,s_2)$ for some holomorphic function $G(s_1,s_2)$. Hence in this case the vanishing orders of $\theta\begin{bsmallmatrix}\ep\\\delta\\\end{bsmallmatrix}(\tau,0)$ with respect to $s_1$ and $s_2$ are both $\frac{1}{8}$.

\item If $\ep_1=0$, then similarly $\theta\begin{bsmallmatrix}\ep\\\delta\\\end{bsmallmatrix}(\tau,0)=\theta\begin{bsmallmatrix}\ep_2&\ep_3\\\delta_2&\delta_3\\\end{bsmallmatrix}(\begin{bsmallmatrix}\tau_2&c\\c&\tau_3\end{bsmallmatrix},0)+O(s_1)+O(s_2)$. By definition \cite{Tan89} of $c=f_{2,3}(s_1,s_2)$, we deduce that $c=0$ when $s_1=s_2=0$, i.e. when the curve hits the boundary $T$. In that case, we have the constant term $\theta\begin{bsmallmatrix}\ep_2&\ep_3\\\delta_2&\delta_3\\\end{bsmallmatrix}(\begin{bsmallmatrix}\tau_2&0\\0&\tau_3\end{bsmallmatrix},0)=\theta\begin{bsmallmatrix}\ep_2\\\delta_2\\\end{bsmallmatrix}(\tau_2,0)\cdot\theta\begin{bsmallmatrix}\ep_3\\\delta_3\\\end{bsmallmatrix}(\tau_3,0)=0$ if and only if $\ep_2=\delta_2=1$. Hence the only theta constants with characteristics that vanish when  $s_1=0$ and $s_2=0$ are $\theta_{33}(\tau,0)$ and $\theta_{37}(\tau,0)$, but by taking partial derivatives one can directly show that neither is divisible by any power of $(s_1\cdot s_2)$.

\end{enumerate}

\begin{proof}[Proof of Proposition \ref{m}]
As in Proposition \ref{h}. we choose the standard boundary component $D_{04}$ so that the two cases $\ep_1=0$ or $1$ correspond to the two orbits of the $\Gamma_g$-action on the pair $(m,04)$. Hence by the same argument as the proof of Proposition \ref{g}, we have
$$
\Omega_{77}(\tau,0)=(s_1\cdot s_2)^{\frac{5}{4}}\cdot F(s_1,s_2)
$$
for some holomorphic function $F(s_1,s_2)$. Moreover, by the expression of $\Omega_{77}$ given in Theorem 0.1, in each summand of $\Omega_{77}$ there is either a $\theta_{33}$ or a $\theta_{37}$ as a factor. Thus
$$
F(s_1,s_2)=\theta\begin{bsmallmatrix}1&1\\1&1\\\end{bsmallmatrix}(\begin{bsmallmatrix}\tau_2&c\\c&\tau_3\end{bsmallmatrix},0)+O(s_1)+O(s_2)
$$
where $\tau_i=f_{i,i}(s_1,s_2)$ ($i=2$ or $3$) and $c=f_{2,3}(s_1,s_2)$ (as in (\ref{l})) are holomorphic functions in $s_1$ and $s_2$, and $c(0,0)=0$. From the discussion directly proceeding the proof, $F(s_1,s_2)$ vanishes when $s_1=0$ and $s_2=0$, and $F(s_1,s_2)$ is not divisible by any power of $(s_1\cdot s_2)$.

The normal direction to $\Delta_0$ in $\CM{3}$ is given by $q=\exp(\pi i\tau_{11})$. $T$ is an irreducible component of the self-intersection of $\Delta_0$, and thus $s_1, s_2$ give the two normal directions. Because the modular form $\Omega_{77}$ vanishes along $T$ with higher order in $s_1,s_2$ than $q$, we conclude that the boundary stratum $T$ is contained in the hyperflex locus $\overline{\HF}$.

\end{proof}

\section{The Catalecticant Hypersurface}

This section is devoted to a discussion of a different locus of plane quartics. In 1868, J. L\"{u}roth discovered that a general plane quartic does not admit a polar pentagon despite a dimension count which suggests that it is possible. A plane quartic that admits a polar pentagon is called a {\em Clebsch quartic}. A detailed discussion can be found in \cite[Sec.~6.3]{Do12}.

Let $E$ be a vector space of dimension $3$, and let $F\in S^4(E^\vee)$ be a degree $4$ homogenous form on $\mathbb{P}(E)$. Consider the apolar map:
$$
ap^2_F:S^2(E)\to S^2(E^\vee), \ [\ldots,x_i\wedge x_j,\ldots]\mapsto \left[\ldots, \frac{\partial^2 F}{\partial x_i\partial x_j},\dots\right].
$$

Fix a basis of $E$ and a dual basis of $E^\vee$, the matrix of $ap^2_F$ in these bases is then called the {\em catalecticant matrix} of $F$, and is denoted by $\Cat_2(F)$.

It is easy to show \cite[Sec.~6.3.5]{Do12} that a quartic $C$ with the defining equation $F$ is Clebsch if and only if $\det\Cat_2(F)$ vanishes. The locus of such quartics in $\M{3}$ is called the {\em catalecticant hypersurface}, and the class of its closure $[\overline\Cat]\in\Pic(\CM{3})$ was computed in \cite{OS11}:
$$
[\overline\Cat]=56\lambda-6\delta_0-16\delta_1.
$$
We give a straightforward alternative derivation of this using modular forms.

Indeed, by the results of \cite{DPFSM}, a plane quartic can be written as in Equation \eqref{p} in terms of an Aronhold system of its bitangents. As in the case of the hyperflex, we discard the coefficients consisting of products of theta constants because they only vanish on some components of the hyperelliptic locus. Set
$$
F=(af)^2+(be)^2+(cd)^2-2afbe-2becd-2afcd,
$$
where $a,b,c,d,e,f$ are as in equation (\ref{p}).
\begin{prop}
With the above $F$, $\det\Cat_2(F)$ is the modular form whose zero locus in ${\mathcal A}_3$ is equal to the image of the locus of Clebsch quartics under the Torelli map.
\end{prop}
\begin{proof}
We only need to show that $\det\Cat_2(F)$ is a modular form. With the identification $E\simeq H^0(C,K_C)^\vee$ we take the standard basis $w_1,w_2,w_3$ in $E$. $F$ has coefficients consisting of theta constants and theta gradients, as does $\det\Cat_2(F)$. Since the theta constants and gradients are modular forms, $\det\Cat_2(F)$ is a modular form.
\end{proof}
\begin{cor}
The expression $\det\Cat_2(F)$ is a modular form of weight $56$, and thus $[\Cat]=56\lambda\in\Pic{\mathcal M}_3$.
\end{cor}
\begin{proof}
To compute the weight, note that in the basis $w_1$, $w_2$, $w_3$, each of $a,b,c,d,e,f$ has coefficients which are products of $3$ theta constants (of weight $\frac{1}{2}$) and one theta gradient (of weight $\frac{5}{6}$). Thus each of the coefficients of $F$ consists of $12$ theta constants and $4$ theta gradients. To compute $\det\Cat_2(F)$ we need to take the coefficients of the six second-order derivatives of $F$ (the columns of $\Cat_2(F)$). Note that the coefficients of second-order derivatives also consist of $12$ theta constants and $4$ theta gradients, thus each summand of the determinant of the $6\times 6$ catalecticant matrix is a product of $72$ theta constants and $24$ theta gradients. With the aid of a computer we can confirm that there are no common factors among the entries of $\det\Cat_2(F)$. Hence the weight of $\det\Cat_2(F)$ is $72\times 1/2+24\times 5/6=56$.
\end{proof}

Implementing the vanishing orders of theta constants and theta gradients as computed in Section 3, we also obtain with the aid of a computer the vanishing orders of $\det \Cat_2(F)$ on $D_n$ and $P_V$, resulting in a computation of the coefficients of $\delta_0$ and $\delta_1$ in the class $[\overline\Cat]$.


\begin{thebibliography}{AMRT10}

\bibitem[ACG11]{ACGII}
E.~Arbarello, M.~Cornalba, and P.~Griffiths.
\newblock {\em Geometry of algebraic curves. {V}olume {II}}, volume 268 of {\em
  Grundlehren der Mathematischen Wissenschaften [Fundamental Principles of
  Mathematical Sciences]}.
\newblock Springer, Heidelberg, 2011.
\newblock With a contribution by Joseph Daniel Harris.

\bibitem[AMRT10]{AMRT}
A.~Ash, D.~Mumford, M.~Rapoport, and Y.-S. Tai.
\newblock {\em Smooth compactifications of locally symmetric varieties}.
\newblock Cambridge Mathematical Library. Cambridge University Press,
  Cambridge, second edition, 2010.
\newblock With the collaboration of Peter Scholze.

\bibitem[{Che}15]{Che15}
D.~{Chen}.
\newblock {Degenerations of Abelian Differentials}.
\newblock {\em ArXiv e-prints} 1504.01983, 2015.

\bibitem[{CS}12]{CS12}
F.~{Claudio}, P.~Stefano.
\newblock {\em An affine open covering of $\M{g}$ for $g\leq 5$}.
\newblock Geom. Dedicata 158 (2012), 61--68.

\bibitem[Cuk89]{Cu89}
F.~Cukierman.
\newblock Families of {W}eierstrass points.
\newblock {\em Duke Math. J.}, 58(2):317--346, 1989.

\bibitem[DPFSM14]{DPFSM}
F.~{Dalla Piazza}, A.~{Fiorentino}, and R.~{Salvati Manni}.
\newblock {Plane quartics: the matrix of bitangents}.
\newblock {\em ArXiv e-prints} 1409.5032, 2014.

\bibitem[Dol12]{Do12}
I.~Dolgachev.
\newblock {\em Classical algebraic geometry}.
\newblock Cambridge University Press, Cambridge, 2012.
\newblock A modern view.

\bibitem[Fay73]{Fay73}
J.~Fay.
\newblock {\em Theta functions on {R}iemann surfaces}.
\newblock Lecture Notes in Mathematics, Vol. 352. Springer-Verlag, Berlin-New
  York, 1973.

\bibitem[GH12]{GrH12}
S.~Grushevsky and K.~Hulek.
\newblock The class of the locus of intermediate {J}acobians of cubic
  threefolds.
\newblock {\em Invent. Math.}, 190(1):119--168, 2012.

\bibitem[{Gru}07]{Gru07}
S.~{Grushevsky}.
\newblock {Geometry of A\_g and Its Compactifications}.
\newblock {\em ArXiv e-prints} 0711.0094, 2007.

\bibitem[HS02]{HS02}
K.~Hulek and G.~Sankaran.
\newblock The geometry of {S}iegel modular varieties.
\newblock In {\em Higher dimensional birational geometry ({K}yoto, 1997)},
  volume~35 of {\em Adv. Stud. Pure Math.}, pages 89--156. Math. Soc. Japan,
  Tokyo, 2002.

\bibitem[Igu72]{Igu72}
J.-I. Igusa.
\newblock {\em Theta functions}.
\newblock Springer-Verlag, New York-Heidelberg, 1972.
\newblock Die Grundlehren der mathematischen Wissenschaften, Band 194.

\bibitem[Igu81]{Igu81}
J.-I. Igusa.
\newblock On the nullwerte of {J}acobians of odd theta functions.
\newblock In {\em Symposia {M}athematica, {V}ol. {XXIV} ({S}ympos., {INDAM},
  {R}ome, 1979)}, pages 83--95. Academic Press, London-New York, 1981.


\bibitem[OS11]{OS11}
G.~Ottaviani and E.~Sernesi.
\newblock On singular {L}\"uroth quartics.
\newblock  {\em Sci. China Math.}, volume~54, page 1757--1766, 2011

\bibitem[SM83]{SM83}
R.~{Salvati Manni}.
\newblock On the nonidentically zero {N}ullwerte of {J}acobians of theta
  functions with odd characteristics.
\newblock {\em Adv. in Math.}, 47(1):88--104, 1983.

\bibitem[SM94]{SM94}
R.~Salvati~Manni.
\newblock Modular varieties with level {$2$} theta structure.
\newblock {\em Amer. J. Math.}, 116(6):1489--1511, 1994.

\bibitem[Tan89]{Tan89}
M.~Taniguchi.
\newblock Pinching deformation of arbitrary {R}iemann surfaces and variational
  formulas for abelian differentials.
\newblock In {\em Analytic function theory of one complex variable}, volume 212
  of {\em Pitman Res. Notes Math. Ser.}, pages 330--345. Longman Sci. Tech.,
  Harlow, 1989.

\bibitem[Tan91]{Tan91}
M.~Taniguchi.
\newblock On the singularity of the periods of abelian differentials with
  normal behavior under pinching deformation.
\newblock {\em J. Math. Kyoto Univ.}, 31(4):1063--1069, 1991.

\bibitem[Ver83]{Ver83}
A.~Vermeulen.
\newblock {\em Weierstrass points of weight two on curves of genus three}.
\newblock Universiteit van Amsterdam, Amsterdam, 1983.
\newblock Dissertation, University of Amsterdam, Amsterdam, 1983, With a Dutch
  summary.

\bibitem[Yam80]{Yam80}
A.~Yamada.
\newblock Precise variational formulas for abelian differentials.
\newblock {\em Kodai Math. J.}, 3(1):114--143, 1980.

\end{thebibliography}

\end{document}